\documentclass[12pt]{amsart}

\usepackage{eucal,color,tikz,graphicx}
\usetikzlibrary{decorations.text}

\author[De Leenheer]{Patrick De Leenheer}
\address{University of Florida, PO Box 118105, Gainesville, 
  FL~32611--8105.}
\email{deleenhe@ufl.edu}
\urladdr{http://www.math.ufl.edu/~deleenhe/}

\author[Gopalakrishnan]{Jay Gopalakrishnan}
\address{Portland State University, PO Box 751, Portland, OR~97207--0751.}
\email{gjay@pdx.edu}
\urladdr{http://web.pdx.edu/~gjay/}

\author[Zuhr]{Erica Zuhr}
\address{University of Florida, PO Box 118105, Gainesville, 
  FL~32611--8105.}
\email{ericaz@math.ufl.edu}
\urladdr{http://www.math.ufl.edu/~ericaz/}

\title[Pattern formation]{Instability in a generalized Keller-Segel Model}

\thanks{This work was partially supported by the NSF under grants
  DMS-0818050, DMS-1014817 and DMS-1211635. PDL wishes to thank the VLAC, the Flemish Academic Centre for Science and the Arts, for hosting and supporting him during a sabbatical leave from the University of Florida, the Universit{\'e} Catholique de Louvain-la-Neuve for awarding him a  visiting professorship in the fall of 2011, and the University of Florida for a Faculty Enhancement Opportunity (FEO) fund.}

\newtheorem{theorem}{Theorem}[section]
\newtheorem{lemma}[theorem]{Lemma}
\newtheorem{corollary}[theorem]{Corollary}
\newtheorem{proposition}[theorem]{Proposition}

\theoremstyle{definition}
\newtheorem{definition}[theorem]{Definition}
\newtheorem{example}[theorem]{Example}

\theoremstyle{remark}
\newtheorem{remark}[theorem]{Remark}

\begin{document}

\newcommand{\com}[1]{{#1}}
\newcommand{\RRR}{\mathbb{R}}
\newcommand{\RNN}{\mathbb{R}^{N\times N}}
\newcommand{\vx}{\vec{x}}
\newcommand{\vy}{\vec{y}}
\newcommand{\vv}{\vec{v}}
\newcommand{\vV}{\vec{V}}
\newcommand{\vg}{\vec{g}}
\newcommand{\valpha}{\vec{\alpha}}
\newcommand{\vkappa}{\vec{\kappa}}
\newcommand{\vkappat}{\vec{\kappa}^{\,t}}
\newcommand{\vs}{\vec{v}_*}
\newcommand{\vsi}[1]{\vec{v}_*^{\,({#1})}}
\newcommand{\vsii}[1]{v_{*,{{#1}}}}
\newcommand{\us}{u_*}
\newcommand{\ve}{\vec{e}}
\newcommand{\Frechet}{Fr{\'{e}}chet}
\newcommand{\grad}{\nabla}

\def\d{\partial}

\begin{abstract}
  We present a generalized Keller-Segel model where an
arbitrary number of chemical compounds react, some of which are
  produced by a species, and one of which is a chemoattractant for the
  species. To investigate the stability of homogeneous stationary
  states of this generalized model, we consider the eigenvalues of a
  linearized system.  We are able to reduce this infinite dimensional
  eigenproblem to a parametrized finite dimensional eigenproblem. By
  matrix theoretic tools, we then provide easily verifiable sufficient
  conditions for destabilizing the homogeneous stationary states. In
  particular, one of the sufficient conditions is that the
  chemotactic feedback is sufficiently strong. Although this mechanism was already known to exist in the original Keller-Segel model, 
  here we show that it is more generally applicable by significantly enlarging the class of models exhibiting this 
  instability phenomenon which may lead to pattern formation.
\end{abstract}



\maketitle

\section{Introduction}

Pattern formation is a fascinating area of biology with many
unanswered questions. An early example is furnished by morphogenesis,
which focuses on the following question: How can an initially
spherically symmetric system, such as an embryo in early stages of
development, lead to a creature as spherically asymmetric as a person,
or an animal with a spotted coat?  One of the first attempts to
theoretically explain morphogenesis (by means of a mathematical model)
is by Alan Turing. In~\cite{turing}, he proposed that an
instability mechanism in specific systems of chemical
reaction-diffusion systems is the underlying cause of morphogenesis
by showing that under certain conditions on both the reaction and
diffusion terms, some such systems which are stable {\em without}
diffusion are destabilized {\em with} diffusion~\cite{maini}.  Over
time, pattern formation in various biological systems has been
attributed to such Turing instabilities in systems of
reaction-diffusion equations.

The main goal of this paper is to present an alternate avenue
potentially leading to pattern formation via chemotaxis.  
The key idea of this mechanism is already present in the 
well-known Keller-Segel model~\cite{ks} (reviewed below), as elucidated by Schaaf~\cite{schaaf}. It relies on the destabilization via a chemotactic term 
of an otherwise stable uniform steady state. \com{This mechanism of destabilization based on the chemotactic sensitivity is also analyzed for the case of the two equation system in one spatial dimension in~\cite{wangxu2012}.} In this paper, we will identify a larger class of systems \com{in higher spatial dimensions} where 
a similar chemotactic destabilization mechanism may lead to pattern formation.

Chemotaxis
occurs when the movement of a species is influenced by chemicals in
the environment~\cite{horstmann}. It naturally arises in a wide
variety of interesting biological settings, including cancer
metastasis, angiogenesis, immune system function and egg
fertilization~\cite{hillen}.  Considering reaction diffusion systems
with chemotactic terms, we show that there are quite general classes
of chemical reaction networks (CRNs) which can be destabilized by
means of an inherent feedback mechanism.  The feedback arises as
follows. A population of cells produces a chemical, which in turn
participates in a CRN. The CRN leads to one or more chemical products.
One of these products serves as the chemotactic signal for the cells,
thereby completing the feedback loop.  Under certain conditions this
feedback loop has unstable homogeneous steady states which, in the
absence of the feedback, would be stable. The destabilization of
homogeneous stationary states, caused by the feedback mechanism built into the system, may give rise to 
pattern formation. As mentioned earlier, such a mechanism is already present in a simple CRN
appearing in the well-known Keller-Segel model~\cite{ks}. 
In this model, the signal emitted by the cells
is the same as the chemotactic signal, and it is subject only  to
decay. Since this CRN consists of a single decay reaction, it is
perhaps one of the simplest possible CRNs.  In this paper, we find a
larger class of CRNs with destabilizable homogeneous steady states.

To further explain the difference with the traditional destabilization
mechanism, let us first review Turing's diffusion driven instability 
for a two-species system 
(generalizations to multi-species systems are also well-known),
as explained in~\cite{murray}, using a reaction-diffusion system of
the form
\begin{subequations} \label{eq:1}
\begin{align}\label{ex1}
u_t & = d_1 \Delta u + f(u,v), \qquad & x \in \Omega, \quad t>0 \\
\label{ex2}
v_t & = d_2 \Delta v + g(u,v), \qquad & x \in \Omega, \quad t>0 \\
\label{ex3}
& \frac{\partial u}{\partial n} = \frac{\partial v}{\partial n} = 0, \qquad & x \in \partial \Omega, \quad t>0.
\end{align}
\end{subequations}
where $d_1$ and $d_2$ are positive diffusion coefficients.  Here $u_t$
denotes $\partial u/\partial t$ and $n$ (in $\partial /\partial n =
n\cdot \nabla$) denotes the outward unit normal on the smooth domain boundary
$\d\Omega$.  Assume the existence of a homogeneous steady state
$(u^*,v^*)$ of~\eqref{eq:1}.  When the diffusion terms are ignored, we
obtain the ordinary differential equation system corresponding to just
the reaction terms:
\begin{subequations}\label{eq:ode1}
\begin{eqnarray}
\frac{du}{dt} =f(u,v), \qquad &  t>0 \\
\frac{dv}{dt} = g(u,v), \qquad & t>0.
\end{eqnarray}
\end{subequations}
We note that $(u^*,v^*)$ is also a steady state of~\eqref{eq:ode1} and
we assume that it is linearly stable, i.e., we assume that 
the Jacobian matrix
\[
J =
\begin{bmatrix}
f_u(u^*,v^*) & f_v(u^*,v^*) \\
g_u(u^*,v^*) & g_v(u^*,v^*)
\end{bmatrix}
\]
has eigenvalues with negative real part. (Again, derivatives are indicated by subscripts.)

The question Turing asked is whether $(u^*,v^*)$, considered as a
steady state of~\eqref{eq:1}, can be linearly unstable, even if it is a
stable steady state of~\eqref{eq:ode1}.  In other words, is it
possible that the eigenvalue problem arising from
linearizing~\eqref{eq:1} at $(u^*,v^*)$, namely
\begin{eqnarray*}
d_1 \Delta U + f_u(u^*,v^*)U + f_v(u^*,v^*)V =& \lambda U \\
d_2 \Delta V + g_u(u^*,v^*)U + g_v(u^*,v^*)V =& \lambda U \\
\frac{\d U}{\d n} = \frac{\d V}{\d n}  =& 0,
\end{eqnarray*}
has an eigenvalue $\lambda$ with positive real part?  It turns out
that a first necessary condition for this to happen is that $d_1\neq
d_2$, implying that both species must move with different diffusion constants. By an appropriate scaling we may assume, without loss of
generality, that $d_1=1$, and then a second necessary condition is
that $J$ has one of the sign patterns
\begin{equation}\label{spat}
\begin{bmatrix} + & - \\ + & -\end{bmatrix}, \quad \begin{bmatrix} - & + \\ - & +\end{bmatrix}, \quad
\begin{bmatrix} + & + \\ - & -\end{bmatrix}, \quad \text{or} \quad \begin{bmatrix} - & - \\ + & +\end{bmatrix},
\end{equation}
with additional restrictions on $d_2$ and $\Omega$. The relevance of this sign pattern is that it imposes certain restrictions on how 
the species interact. For example, the first matrix implies that the first species activates itself and the second species, and 
that the second species inhibits itself,  as well as the first species. \com{A generalization of Turing-type instabilites to a system modeling an arbitrary number of chemicals can be found in~\cite{satnoianu2000}.}

The model that we propose for pattern formation on the other hand, is a fully nonlinear
reaction-diffusion system which also contains a nonlinear chemotaxis term.
Destabilization of homogeneous steady states in our model does {\em not}
require any restrictions on the diffusion coefficients, although we do
impose some conditions on the sign pattern of the Jacobian
corresponding to some of the reaction terms. These may however be less
crucial than the role played by the chemotactic term.

The Keller-Segel model for chemotaxis is one of the
most widely studied models. It was developed focusing on the cellular
slime mold \textit{Dictyostelium discoideum}. As explained in
\cite{horstmann} and \cite{ks}, the derivation of the Keller-Segel
model is based on the aggregation stage of \textit{D. discoideum}'s
life cycle. The unicellular slime mold grows by cell division until it
depletes its food source, and begins to enter a starvation mode. After
some time, one cell will emit a signal of cyclic Adenosine
Monophosphate (cAMP). The other cells are chemotactically attracted to
this signal, and in turn begin to emit cAMP themselves. The cells also
produce an enzyme which degrades the cAMP by first binding to it and
forming a complex which in turn breaks up into the enzyme plus what we
call a ``degraded product''.  To model this, the following assumptions
are made in~\cite{horstmann} and~\cite{ks}:

\begin{itemize}

\item Let $u(x,t)$ denote the cell density of the slime mold, $v(x,t)$
  the density of the chemoattractant cAMP, $\eta(x,t)$ the density of the
  enzyme, $c(x,t)$ the density of the complex formed by the binding
  of cAMP and the enzyme, and $d(x,t)$ denote the density of the degraded product.

\item The chemoattractant is produced by the amoeba at a rate $f(v)$
  per amoeba and the enzyme is produced at a rate 
  $g(v,\eta)$ per amoeba, which, according to \cite{ks} 
  is allowed to vary with both the chemoattractant and enzyme concentrations.

\item The chemoattracant, enzyme and complex react as shown below.
  \[
  v + \eta \;\longleftrightarrow\; c \;\longrightarrow\;
  \eta +  d.
  \]
  We assume that the reaction rates proceed according to the law of mass
  action. Accordingly, let the positive rate constants be denoted by
  $r_1$ (for the reaction $v+\eta \rightarrow c$), $r_{-1}$ (for the
  reaction $c \rightarrow v+\eta$) and $r_2$ (for the reaction $c
  \rightarrow \eta + d$).

\item The cAMP, the enzyme and the complex all diffuse according to
  standard Fick's law.

\item The cell concentration changes by random diffusion, as well by
  chemotaxis, due to the cells ``climbing the gradient'' of the
  chemoattractant cAMP.

\end{itemize}
The last assumption is again modeled using Fick's law of diffusion,
but with the amoebic flux $J^{(u)}$ consisting of a purely diffusive
part and a chemotactic contribution, namely
\[
J^{(u)}(x,t) = k_2 \nabla v - k_1 \nabla u,
\] 
for some chemotactic sensitivity function $k_2(u,v)$ and diffusion
constant $k_1$. These assumptions lead to the full Keller-Segel
model~\cite{horstmann}  on a bounded domain $\Omega$ for time $t > 0$:
\begin{subequations}
  \label{eq:FullKS}
\begin{align} \label{genks1}
		u_{t} &= \nabla \cdot (k_1 \nabla u - k_2(u,v)\nabla v), & x  \in  \Omega, \quad t>0 \\
		\label{genks2}
		v_{t} &= k_v \Delta v - r_1 v \eta + r_{-1} c + u f(v),  & x  \in  \Omega, \quad t>0\\
		\label{genks3}
		\eta_{t} &= k_{\eta} \Delta \eta -r_1 v \eta +(r_{-1}+r_2) c + u g(v,\eta),  & x \in  \Omega, \quad t>0\\
		\label{genks4}
		c_{t} &=  k_{c} \Delta c + r_1 v \eta - (r_{-1}+r_2)c,  & x \in  \Omega, \quad t>0 \\
		\label{genks5}
		& \frac{\partial u}{\partial n} = \frac{\partial v}{\partial n} = \frac{\partial \eta}{\partial n} =\frac{\partial c}{\partial n} = 0, & x \in \partial \Omega, \quad t>0,
\end{align}
\end{subequations}
Here, $k_v,k_{\eta}$ and $k_{c}$ denote the positive diffusion constants for cAMP, enzyme and complex respectively.

In~\cite{ks} Keller and Segel use a steady state assumption to reduce
the above set of four equations to a two-equation system which models
only the density of the chemoattractant and the cell density. Further
simplifications result in the so-called~\cite{hillen,horstmann}
``minimal system'',
\begin{subequations}
  \label{eq:minimalKS}
  \begin{align}\label{min1}
		u_{t} &= \nabla \cdot (\nabla u - \chi u \nabla v), & x  \in  \Omega, \quad t>0 \\ 
		\label{min2}
		v_{t} &= k_v \Delta v - \gamma v + \alpha u,  & x  \in  \Omega, \quad t>0\\
		\label{min3}
		& \frac{\partial u}{\partial n} = \frac{\partial v}{\partial n} = 0, & x \in \partial \Omega, \quad t>0 \\
		\label{min4}
		& u(0,x) = u_0(x), \quad v(0,x)=v_0(x), & x \in \Omega.
  \end{align}		
\end{subequations}
This system has been studied extensively, and there are many results
on local and global existence, positivity, and blow-up of solutions as
well as steady
states~{\cite{perthame-book,hillen,horstmann,schaaf}}.
Solutions that exhibit blow-up, defined as solutions for which the
$L^{\infty}$-norm of $u$ or $v$ becomes unbounded in either finite or
infinite time, have attracted considerable interest and are reviewed
extensively in~\cite{horstmann}.  (The blow-up phenomenon, restricted
to a simplified Keller-Segel model, is also featured
in~\cite[Ch.~5]{perthame-book}, which nicely illustrates the
difficulties encountered when dealing with this complex issue in an
elementary setup.)
One significant result is that in two
dimensions, if $\alpha \chi \int_{\Omega} u_0(x)dx < 4 \pi k_v$ then
the solution to~\eqref{eq:minimalKS} exists globally in time and its
$L^{\infty}$-norm is uniformly bounded for all time~\cite[Table
5]{horstmann}. On the other hand if $4 \pi k_v < \alpha \chi
\int_{\Omega} u_0(x) dx < 8 \pi k_v$, then there exist initial data
$(u_0,v_0)$ for which the solution of~\eqref{eq:minimalKS} blows up at
the boundary of $\Omega$ in finite or infinite time~\cite[Table
5]{horstmann}. Although the question of blow-up of solutions for our
generalization of the Keller-Segel model is interesting, we do not
address that issue in this paper. Instead, here we focus exclusively
on destabilization of homogeneous steady states.

{Our work fits in the context of the more recent attempts to investigate multi-species, multi-chemical systems with chemotactic terms 
\cite{perthame-paper,espejo,fasano,smith,wolansky,horstmann11,liu2012}. Just as in the case of the Keller-Segel model, the 
blow-up phenomenon in these generalized models has received considerable attention \cite{wolansky,perthame-paper,espejo,horstmann11}. 
Of particular interest to the issue we address here, is the recent work by Fasano et al \cite{fasano}, Horstmann \cite{horstmann11},  \com{and Liu et al \cite{liu2012}}. 
All papers consider the existence of non-homogeneous or periodic steady states, a signature of pattern formation for multi species models.
In \cite{fasano}, a two-species two-chemical system is proposed where each species produces one of the two chemicals, 
which in turn serve as attractive chemotactic signals for the species that produce them. In addition however, these chemicals are also chemotactic repellants for the species that do not produce them. The model is analyzed for scalar spatial domains, and focuses on the existence of nonuniform and periodic  
steady state patterns. Motivated in part by \cite{fasano}, Horstmann in \cite{horstmann11} proposed an $n$ species, $m$ chemicals model 
where all chemicals are allowed to serve as chemotactic signals for the species, and the species themselves can chemotactically attract or repel 
each other. In addition, the chemicals may be produced and/or consumed by the species, and in principle, could react with each other in fairly 
general ways. Several natural questions that have attracted the attention of several researchers in the context of the Keller-Segel model, are 
addressed in \cite{horstmann11} for the multi-species, multi-chemical models. They include blow-up, global existence of solutions, existence and non-existence of nonuniform steady states, destabilization of uniform steady states as a way to achieve pattern formation, 
and the existence of Lyapunov functionals. 
Although~\cite{horstmann11} begins with this very general model, the focus of the paper quickly turns to 
specific cases, motivated by the literature, where only a limited number of non-reacting chemicals (often just one or two) appear in the model. One 
notable exception is in Section 3.1 of \cite{horstmann11}, where a one-species, two chemical system is proposed, where the two chemicals 
react in a nontrivial way. The model is inspired by work of \cite{boon} where an {\it E.~coli} population is studied in the presence of two 
chemotactic attractants, namely glucose and oxygen. \com{In~\cite{liu2012} a system with two chemical agents is also considered, but in this case one is a chemoattractant while the other acts as a chemorepellent.}
In contrast, the emphasis of the work presented here, is on possible pattern formation in models with several chemicals, which react 
according to fairly general reaction networks, but only a single species that chemotactically responds to one of these chemicals.}

In Section~\ref{sec:general-model} of this article, we explain our
generalized model and the system of partial differential equations
which arises from it. In Section~\ref{sec:homog}, we discuss
conditions for the existence of steady state solutions to the system
in a special case. In Section~\ref{sec:instab} we state our main
results, which give sufficient conditions under which steady states
are unstable. Section~\ref{sec:examples} gives examples of theoretical
biological systems which meet the conditions of our theorems and also
provides a motivating example for generalizing our results.

\section{The generalized model} \label{sec:general-model}

In this section we present a generalization of the Keller-Segel
model~\eqref{eq:minimalKS} where the modeled species interacts with
several chemical compounds.

As before, consider a population of some species occupying $\Omega$,
an open bounded connected $n$-dimensional domain. Let the population
density at a point $x$ in $\Omega$ and time $t$ be $u(x,t)$.  The population climbs
the gradient of a chemical of concentration $v_N$. To model the more
general and the more likely biological scenario, where this chemical
is in reaction with several other compounds in the environment, we
introduce the concentrations of $N-1$ other chemicals, denoted by
$v_1,v_2, \ldots, v_{N-1}$ (see Figure~\ref{fig:sch}).  These
chemicals can interact with each other as well as with $v_N$ creating
a CRN.  Writing $\vec{v} = [v_1, \ldots, v_N]^t$, we model the
rate of change of $\vec{v}$ due to the CRN by $\vec{g}(\vec{v})$ for
some function $\vec{g}: \mathbb{R}^{N} \to \mathbb{R}^{N}$.
In applications,
this function $\vec{g}$ is often determined by mass action
kinetics. Note that in general $\vec g$ includes decay of the
chemicals.

Some of these compounds can also be produced by the species.  To model
this scenario, consider a general subset $S \subseteq \{1,2, \ldots,
N\}$ and assume that the species produce the chemical $v_l$ at a rate
$\alpha_l \geq 0$ for all $l \in S$.  The remaining compounds are not
generated by the species, so we set $\alpha_l=0$ for $l\not\in S$. Let
$\valpha \in \mathbb{R}^N$ be the vector whose $i$th component is the
$\alpha_i$ introduced above.  Thus, the rate of change of $\vec v$
would be governed by $\vec v_t = \valpha u + \vec g (\vec v)$, in the
absence of any diffusive effects.

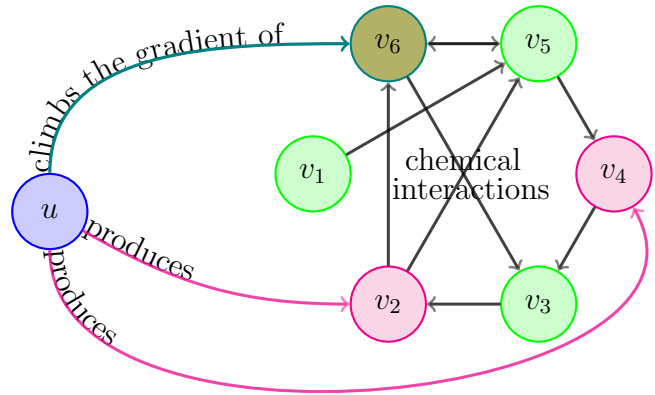
\begin{figure}[h]
  \centering
  \begin{tikzpicture} 
    [produce/.style={black, opacity=0.5, very thick},
    species/.style ={circle,minimum size=1cm,draw=blue,fill=blue!20,thick}, 
    compound0/.style={circle, minimum size=1cm,draw=green!50!blue,fill=red!50!green!60,thick},
    compound1/.style={circle, minimum size=1cm,draw=green,fill=green!20,thick},
    compound2/.style={circle, minimum size=1cm,draw=magenta,fill=magenta!20,thick}]
     
    %

     \node at (  0:2) [compound2] (v4)  {$v_4$}; 
     \node at ( 60:2) [compound1] (v5)  {$v_5$}; 
     \node at (120:2) [compound0] (v6)  {$v_6$}; 
     \node at (180:2) [compound1] (v1)  {$v_1$}; 
     \node at (240:2) [compound2] (v2)  {$v_2$}; 
     \node at (300:2) [compound1] (v3)  {$v_3$}; 

     \node at (-5.5,-0.5) [species]  (u)  {$u$};

     \draw [->,green!50!blue,very thick, out=90, in=180,preaction={decorate,
       decoration={text along path,text={|\rmfamily|climbs the gradient of}}}]
     (u)  to (v6);

    \draw [->,produce,magenta,out=-30,in=180, preaction={decorate,
            decoration={text along path,text={|\rmfamily|produces}}}]
          (u)  to (v2);
     \draw [->,produce,magenta,out=-90,in=-60,preaction={decorate,
            decoration={text along path,text={|\rmfamily|produces}}}]         
          (u)  to  (v4);

     \draw[->,produce] (v1) -- (v5);
     \draw[->,produce] (v6) -- (v3);
     \draw[->,produce] (v2) -- (v5);
     \draw[->,produce] (v4) -- (v3);
     \draw[->,produce] (v5) -- (v4);
     \draw[->,produce] (v3) -- (v2);
     \draw[->,produce] (v2) -- (v6);
     \draw[<->,produce] (v6) -- (v5);
     
     \node at (0,0.2)  {chemical};
     \node at (0.1,-0.2) {interactions};
  \end{tikzpicture}
  \caption{A schematic of the species $(u)$ and the chemicals
    ($v_i$'s) interacting.}
  \label{fig:sch}
\end{figure}

However, we also want to account for spatial diffusion. In general,
both the species population and the chemical concentrations diffuse.
Assume that the diffusion coefficient for $u$ is given by the constant
$D>0$, and the diffusion coefficients for each $v_i$ are given by the
constants $\tilde{D}_i>0$. Let $\tilde{D}$ be the diagonal matrix
whose $i$th diagonal entry is $\tilde{D}_i$. Assuming that the
chemotactic sensitivity function is $\chi u$, for some constant $\chi
> 0$, we arrive at our generalized model:
\begin{subequations}
  \label{eq:GKS}
\begin{align}                   \label{gen1}
  u_t 
  &= \nabla \cdot ( D \nabla u - \chi u \nabla v_N) 
  && x \in \Omega, \quad t >0, 
  \\             \label{gen2}
  \vec{v}_t 
  & = \tilde{D} \Delta \vec{v} +\valpha u + \vec{g}(\vec{v}) 
  && x\in \Omega, \quad t >0,
  \\            \label{gen3}
  u(x,0) & = u_0(x), \qquad \vec{v}(x,0) = \vec{v}_0(x) 
  && x \in \Omega, 
  \\ \label{gen4}
  \frac{\partial u(x,t)}{\partial n} 
  & = \frac{\partial v_i(x,t)}{\partial n} = 0 
  && x \in \partial \Omega,\quad t \geq 0, \quad  1 \leq i \leq N.
\end{align}
\end{subequations}
Here we have assumed zero Neumann boundary conditions. The initial
conditions $u_0$ and $v_0$ are assumed to be nonnegative functions
on $\Omega$. \com{Notice that upon integrating~\eqref{gen1} over the domain and applying the divergence theorem and boundary conditions,  we have
\begin{align*}
\int_{\Omega} u_t \, dx &= \int_{\Omega} \nabla \cdot (D \nabla u - \chi u \nabla v_N) \, dx \\
&= \int _{\d \Omega} (D \nabla u - \chi u \nabla v_N) \cdot n \, dS \\
&=0.
\end{align*}
Hence if $u$ is sufficiently smooth we may interchange the integral with the differentiation with respect to time to conclude that the total population of the amoeba inside the domain is constant in time.}

As an immediate example of an application of our generalized model, consider the unreduced Keller-Segel model~\eqref{eq:FullKS} with the chemotactic sensitivity function $k_2(u,v) = \chi u$ and the rates $f(v)$ and $g(v,\eta)$ assumed to be constant. We change the notation to let $v_1$ denote the density of the emitted enzyme, $v_3$ the density of the chemoattractant cAMP and $v_2$ the density of complex formed from the cAMP and the enzyme. Then setting
\[
\vg(\vv) = \begin{bmatrix}
- r_1 v_1 v_3 + (r_{-1} + r_2)v_2 \\
r_1 v_1 v_3 - (r_{-1}+r_2)v_2 \\
- r_1 v_1 v_3 + r_{-1} v_2 \\
\end{bmatrix},
\]
and
\[
\alpha = \begin{bmatrix} \alpha_1 \\ 0 \\ \alpha_3 \end{bmatrix}
\]
for constant rates $\alpha_i > 0$, we recover~\eqref{eq:FullKS} in the framework of our model.


In the following sections, we will examine the constant stationary
states of the model~\eqref{eq:GKS}. We will be especially concerned
with conditions for the destabilization of such constant states.

\section{Existence of homogeneous steady states}   \label{sec:homog}

A steady state $(\us,\vs)$ of our generalized model~\eqref{eq:GKS}
satisfies
\begin{subequations}
  \label{eq:ss}
\begin{eqnarray}\label{ssa}
D \Delta \us - \nabla \cdot (\chi \us \nabla [\vs]_N) &=& 0, 
\quad x \in \Omega,
\\ 
\label{ssb}
\tilde{D} \Delta\vs +\vec{\alpha} \us + \vec{g}(\vs) &=& 0, 
\quad x \in \Omega, \\
\label{ssc}
\frac{\d \us}{\d n} = \frac{\d \vs}{\d n} &=& 0, \quad x \in \d \Omega.
\end{eqnarray}
\end{subequations}
If the steady state is homogeneous, i.e., if it consists of spatially 
constant functions, then~\eqref{ssc} obviously holds, and 
\eqref{ssa}--\eqref{ssb} reduces to
\begin{equation}\label{eq:ss_condition}
\vec{\alpha} \us + \vec{g}(\vs) = 0.
\end{equation}
\com{Recalling that the total population of $u$ is conserved, we 
note that the initial conditions will predetermine the value of $u_*$.} A basic question is whether such nonnegative constant steady states
$(\us,\vs)$ exist.  In this section, we will answer this question in
the affirmative, assuming a certain linearity condition on~$\vg$.

Before we do so, we need to introduce some terminology from matrix
theory that we will use throughout.  For any square matrix $A$, let
$\rho(A)$ denote the spectral radius of $A$. We write $A \geq 0$ for a
real matrix (or vector) $A$ if each entry of $A$ is
nonnegative. Similarly $A>0$ signifies that every entry of $A$ is
positive.

\begin{definition}
  A {\em Metzler matrix} $A \in \RNN$ is a matrix having nonnegative
  off-diagonal entries, i.e.,
  \[
  A_{ij} \geq 0, \qquad i \neq j.
  \]
\end{definition}

We will also need $M$-matrices, which can be defined in many
equivalent ways. Following~\cite{mmats}, we introduce the definition
below.

\begin{definition} 
A matrix $A \in \mathbb{R}^N \times \mathbb{R}^N$ is an \textit{$M$-matrix} if it satisfies
\begin{enumerate}
\item $A_{ij} \leq 0, \qquad i \neq j$
\item $A_{ii} \geq 0$
\item There exists a nonnegative matrix $B\ge 0$ and a number $s \geq
  \rho(B)$ such that $A = sI - B$.
\end{enumerate} 
\end{definition}

Clearly, if $A$ is an $M$-matrix, then $-A$ is a Metzler matrix. 
An $M$-matrix $A$ is invertible if and only if there
exists a $B\ge 0$ and $s> \rho(B)$ such that $A = sI - B$. This is because by the Perron-Frobenius Theorem ~\cite{mmats}, 
the spectral radius of a matrix $B\ge 0$ is also an eigenvalue of the matrix.
We will also need the following well-known result~\cite[pp.~137]{mmats}:

\begin{lemma} \label{lem:equivM}
  Suppose $A \in\RNN$ is nonsingular. Then $A$ is an $M$-matrix if and
  only if~$A^{-1} \geq 0$.
\end{lemma}

We are now ready to address the above mentioned question of existence
of homogeneous steady states, under the assumption that each component
of $\vg$ is a linear function.

\begin{proposition}  \label{prop:linear}
  Suppose $\vg(\vv) = - A \vv$ for some $A$ in $\RNN$ and $\valpha \ge
  0$.  If $A$ is a nonsingular $M$-matrix, then for every positive
  constant $\us$, there exists a unique nonnegative homogeneous
  stationary solution $(\us,\vs)$ of~\eqref{eq:GKS}.  Conversely, $A$
  is a nonsingular $M$-matrix if~\eqref{eq:GKS} has a nonnegative
  homogeneous stationary solution $(\us,\vs)$ for every $\us>0$ and
  every~$\valpha \ge 0$.
\end{proposition}

\begin{proof}
  First, suppose that $A$ is a nonsingular $M$-matrix. Let $\us$ be
  any positive constant. Then setting
  \begin{equation}
    \label{eq:2}
    \vs = \us  A^{-1} \valpha,
  \end{equation}
  we observe that $\vs$ satisfies~\eqref{eq:ss_condition}. Hence
  $(\us,\vs)$ satisfies~\eqref{eq:ss}. Furthermore, $\vs \ge 0$ by
  virtue of Lemma~\ref{lem:equivM}, so $(\us,\vs)$ is a nonnegative
  homogeneous stationary solution. For the given $\us$, the component
  $\vs$ is uniquely determined. Indeed, if $(\us,\vs')$ were another
  homogeneous stationary solution, then by~\eqref{eq:ss_condition}, $
  A\vs' = \us \valpha$. But by the invertibility of $A$
  and~\eqref{eq:2}, this implies $\vs' = \vs$.  This proves the first
  assertion of the proposition.

  Conversely, suppose $(\us,\vsi i)$ is a nonnegative homogeneous
  stationary solution obtained using $\us=1$ and $\valpha = \ve_i$,
  the unit vector in the $i$th coordinate direction.  Then,
  by~\eqref{eq:ss_condition}, $A\vsi i = \ve_i$. Letting $C$ denote
  the matrix whose $i$th column is $\vsi i$, this implies that $AC =
  I$, the identity. Hence $A$ is invertible, and moreover, since $\vsi
  i \ge 0$, we have $A^{-1}= C \ge 0$. Hence by
  Lemma~\ref{lem:equivM}, $A$ is a nonsingular $M$-matrix.
\end{proof}

Having given a sufficient condition under which plenty of nonnegative
homogeneous steady states exist, we now proceed to study the stability
of such states, assuming they exist. Note that we no longer assume
that $\vg$ is linear in the ensuing analysis.

\section{Instability of homogeneous steady states}   \label{sec:instab}

In order to analyze the stability of a given homogeneous steady state
$(\us,\vs)$, we begin by linearizing the spatial partial differential
operator in~\eqref{eq:GKS} about $(\us,\vs)$. Let $\lambda$ be an
eigenvalue of the linearized operator, i.e., $\lambda$ satisfies 
\begin{subequations}
  \label{eq:LinearizedEW}
  \begin{align}\label{lin_ss1}
    D \Delta U - \chi \us  \Delta V_N 
    &= \lambda U , & x  \in  \Omega, \quad t>0
    \\ \label{lin_ss2}
    \tilde{D} \Delta \vV + \valpha U
    + J \vV 
    &= \lambda\vV , & x  \in  \Omega, \quad t>0 
    \\ \label{lin_ss3}
    \frac{\partial U}{\partial n} = \frac{\partial V_k}{\partial n} 
    &= 0, & x \in \partial \Omega, \quad k=1,2,\ldots, N.
\end{align}
\end{subequations}
for some nontrivial pair of functions $(U(x), \vV(x))$ on $\Omega$.
Here the $N\times N$ (Jacobian) matrix $J$ denotes the \Frechet\
derivative of $\vg$ with respect to $\vv$
\begin{equation}
  \label{eq:J}
  J = \frac{\partial \vg}{\partial \vec{v}}\bigg|_{\vs},
\end{equation}
i.e., $J_{kl} = \d g_k / \d v_l$ evaluated at $\vv = \vs$.  The
purpose of this section is to study the eigenvalue
problem~\eqref{eq:LinearizedEW} and thereby find conditions under
which the homogeneous state $(\us,\vs)$ is unstable.

We will use the weak formulation of~\eqref{eq:LinearizedEW}, where $u$
and the components $v_i$ are all sought in the (complex) Sobolev space
$H^1(\Omega)$ consisting of square integrable functions
whose first order distributional derivatives are also square integrable.
By definition of the weak eigenproblem, the number $\lambda$ is
an eigenvalue of the linearized operator if there exists a nontrivial
\com{$(U,V_1,\ldots V_N)$ in $H^1(\Omega)^{N+1}$} satisfying 
\begin{subequations}
  \label{eq:WeakLinEW}
  \begin{align}
    \label{eq:WeakLinEW-1}
    -\int_\Omega D \grad U \cdot \overline{\grad \varphi}
    + \int_\Omega \chi \us \grad V_N \cdot \overline{\grad \varphi }
    & = \lambda \int_\Omega U \overline{\varphi},
    \\ 
    \label{eq:WeakLinEW-2}
    -\int_\Omega \tilde D_k \grad V_k \cdot \overline{\grad \varphi}
    + 
    \alpha_k \int_\Omega U \overline{\varphi }
    +
    \sum_{l=1}^N J_{kl} \int_\Omega V_l \overline{\varphi }
    & = \lambda \int_\Omega V_k \overline{\varphi}, 
  \end{align}
\end{subequations}
for $k = 1,2,\ldots, N,$ and for all $\varphi$ in $H^1(\Omega).$ The integrals are with respect to
the standard Lebesgue measure (omitted). The
system~\eqref{eq:WeakLinEW} can be obtained from the classical
formulation~\eqref{eq:LinearizedEW} by multiplying the equations
of~\eqref{eq:LinearizedEW} by a test function $\overline\varphi$ and
integrating by parts, provided 
the boundary of the domain $\d \Omega$ and 
the eigenfunctions $u, \vec{v}$ 
are smooth
enough. So as to admit irregular domains with possibly nonsmooth
eigenfunctions, we adopt the weak formulation~\eqref{eq:WeakLinEW} as
the definition of our eigenproblem.

\subsection{Reduction to finite dimensions}

The first step in our study of~\eqref{eq:LinearizedEW} proceeds by
generalizing a method of Schaaf~\cite{schaaf}. This allows us to
relate the eigenvalues of the infinite dimensional
eigenproblem~\eqref{eq:WeakLinEW} to a finite dimensional
eigenproblem. The latter is easier to analyze. The finite dimensional
eigenproblem involves an $(N+1) \times (N+1)$ matrix, written in block
form as
\[
M(\mu) = 
\begin{bmatrix}
  D \mu   &  \vkappat\\
  \valpha & J + \mu \tilde D
\end{bmatrix},
\]
where $\vkappat = (0,0,\cdots, -\chi \us \mu)$ and $\mu$ is a real
parameter.  The parameter $\mu$ will always be set to one of the
numbers 
\[
0= \mu_0 > \mu_1 \geq \mu_2 \geq \mu_3 \geq \cdots
\] 
in the
spectrum of the Laplace operator with Neumann boundary conditions,
i.e., $\mu_i$ satisfies
\begin{equation}
  \label{eq:NeumannEW}
  -\int_\Omega \grad \omega_i \cdot \overline{\grad \varphi} = 
  \mu_i \int_\Omega \omega_i \overline{\varphi},
  \qquad \forall \varphi \in H^1(\Omega),
  \quad i=0,1,2,\ldots,
\end{equation}
where $\omega_i$ is the corresponding eigenfunction $\omega_i$ in
$H^1(\Omega)$, normalized to have unit $L^2(\Omega)$-norm.
With these notations, the reduction to finite
dimension is achieved by the following result.


\begin{theorem} \label{thm:reduction} %
  The number $\lambda$ solves the eigenproblem~\eqref{eq:WeakLinEW} if
  and only if it is an eigenvalue of $M(\mu_i)$ for some integer~$i \geq
  0$. 
\end{theorem}
\begin{proof}
  Suppose $\lambda$ satisfies~\eqref{eq:WeakLinEW} with a nontrivial
  $(u,v_1,\ldots v_N)$ in $H^1(\Omega)^{N+1}$.  Then, since $\{
  \omega_i : i \ge 0 \}$ form a complete orthonormal set in
  $L^2(\Omega)$, it is possible to find an $i$ such that not all of
  the $N+1$ numbers
  \[
  x_0 = \int_\Omega u \overline\omega_i, 
  \quad\text{and}\quad 
  x_j = \int_\Omega v_j \overline\omega_i, \quad\text{for } j=1,2,\ldots, N,
  \]
  are zero, i.e., the vector $\vx = [x_0, x_1, \cdots, x_N]^t$ is
  nontrivial. Furthermore, choosing $\varphi = \omega_i$ (with the
  above found $i$) in~\eqref{eq:WeakLinEW},
  using~\eqref{eq:NeumannEW}, and rewriting the resulting equations in
  terms of the numbers $x_j$, we find that
  \begin{equation*}
    \begin{bmatrix}
      D \mu_i & 0 & 0 & \cdots &  -\chi \mu_i \us \\
      \alpha_1 & J_{11} + \mu_i \tilde{D}_1 & J_{12} & \cdots & J_{1N} \\
      \alpha_2 & J_{21} & J_{22} + \mu_i \tilde{D}_2 & \cdots & J_{2N} \\
      \vdots & \vdots & \vdots & \ddots & \vdots  \\
      \alpha_N & J_{N1} & J_{N2} & \cdots & J_{NN} + \mu_i  \tilde{D}_N
    \end{bmatrix}
    \vec{x} = \lambda \vec{x}.
  \end{equation*}
  The matrix above is the same as $M(\mu_i)$. Thus we have shown that
  $\lambda$ is an eigenvalue of $M(\mu_i)$ with $\vx \ne \vec 0$ as
  its corresponding eigenvector.

  To prove the converse, suppose $\lambda$ satisfies 
  \[
  M(\mu_i) \vx = \lambda \vx
  \]
  for some $\vx \ne \vec 0$ and some $i$. Then, setting 
  \[
    U  = x_0 \omega_i
    \quad\text{and}\quad
    V_k = x_k \omega_i\quad \text{ for } k =1,2,\ldots, N,
  \]
  we find that the equations of~\eqref{eq:WeakLinEW} hold for any
  $\varphi \in H^1(\Omega)$. Thus, $\lambda$ is an eigenvalue of the
  linearized operator with the (nontrivial) pair $(u,\vv)$ as its
  corresponding eigenfunction. 
\end{proof}

\begin{remark}
  \com{Recalling that the mean of $u$ in~\eqref{gen1} does not vary
    with $t$, we may restrict ourselves while linearizing to
    perturbations $U$ with $\int_\Omega U =0$. Suppose we also
    restrict the perturbations in $V_k$ to be spatially nonhomogeneous
    by imposing $\int_\Omega V_k =0$. Then, following the proof of
    Theorem~\ref{thm:reduction}, it is easy to see that $\lambda$ is
    an eigenvalue of $M(\mu_i)$ for some $i\ge 1$ if and only if it
    solves the eigenproblem~\eqref{eq:WeakLinEW} with the additional
    conditions $\int_\Omega U = \int_\Omega V_k =0$ for all
    $k=1,\ldots, N.$}
\end{remark}

\subsection{A sufficient condition for destabilization}

We now focus on conditions under which the linearized operator has at
least one eigenvalue $\lambda$ in the right half of the complex
plane. In such a case, the corresponding stationary state is called
{\em linearly unstable}.  In view of Theorem~\ref{thm:reduction}, we
only need to study the eigenvalues of the parametrized
matrix~$M(\mu)$. Let us begin by reviewing a few useful results from
matrix theory that we shall need. The following definitions and the
next two theorems are well-known -- see e.g.~\cite[Theorems~1.4
and~2.35 in Chapter~2]{mmats}.

\begin{definition}
  The \textit{directed graph} $G(A)$ associated to an $N \times N$
  matrix $A$, consists of $N$ vertices $V_1,V_2,...,V_N$ where an edge
  leads from $V_i$ to $V_j$ if and only if $a_{ij} \neq 0$.
\end{definition}
\begin{definition}
  A directed graph $G$ is \textit{strongly connected} if for any
  ordered pair $(V_i,V_j)$ of vertices of $G$ (with $i\ne j$), there
  exists a sequence of edges (a {\em path}) which leads from $V_i$ to~$V_j$.
\end{definition}
\begin{definition}
  A matrix $A$ is called \textit{irreducible} if $G(A)$ is strongly
  connected.
\end{definition}

\begin{theorem}\rm{(Perron-Frobenius)}  \label{thm:perronfrob}
Let $A \ge 0$ be a square matrix.
\begin{enumerate}
\item If $A>0$, then $\rho(A)$ is a simple eigenvalue of $A$, greater
  than the magnitude of any other eigenvalue.

\item If $A$ is irreducible, then $\rho(A)$ is a simple eigenvalue,
  any eigenvalue of $A$ of the same modulus is also simple, $A$ has a
  positive eigenvector $\vx$ corresponding to $\rho(A)$, and any
  nonnegative eigenvector of $A$ is a multiple of $\vx$.
\end{enumerate}
\end{theorem}

\begin{theorem}\rm{(Spectral radius bounds)}    \label{thm:sradius}
  Let $A\ge 0$ be an irreducible $N \times N$ matrix. Letting $s_i$
  denote the sum of the elements of the $i$th row of $A$, define
  \[
  S(A) = \max_{1 \leq i \leq N} s_i
  \qquad\text{and}\qquad 
  s(A) = \min_{1 \leq i \leq N} s_i.
  \] 
  Then the spectral radius $\rho(A)$ satisfies
  \[
  s(A) \leq \rho (A) \leq S(A).
  \]
\end{theorem}

With the help of these well-known results, we can now give a
sufficient condition for the destabilization. 


\begin{theorem} \label{thm:suff1} %
  Assume that the system~\eqref{eq:GKS} has a positive
  homogeneous steady state solution $(\us,\vs)$
  and let $J$ be as in~\eqref{eq:J}.  Suppose $J$ and $\valpha \ge 0$
  satisfy the following conditions:
  \begin{enumerate}
  \item There is an \com{$i_*$ with} $1 \leq i_* \leq N$ such that $\alpha_{i_*} > 0$,
  \item $J$ is irreducible, and 
  \item $J$ is Metzler.
  \end{enumerate}
  Then, if either $\chi \us$ or $\alpha_{i_*}$ is sufficiently large,
  then $(\us,\vs)$ is linearly unstable.
\end{theorem}

\begin{remark}
  The biological interpretation of the first condition is that the
  organism must produce at least one of the chemicals in the reaction
  network. The second condition is a technical condition which we will
  relax in the next theorem. The third condition is the most
  restrictive of the conditions, but we give an example of a large
  class of CRNs for which it holds in Corollary~\ref{cor:example}. We
  also give examples of CRNs which do not meet condition three but for
  which a similar result still holds (see Section~\ref{sec:examples}).
\end{remark}

\begin{remark}
\com{For any fixed $i_*$ we may rescale $u_*$ by $u_*'=\alpha_{i_*} u_*$ to rewrite~\eqref{ssa} as
\[
D \Delta u_*' - \nabla \cdot (C \nabla [\vv_*]_N) =0
\]
where $C=\alpha_{i_*} \chi u_*$. Then the condition
on $\chi u_*$ or
$\alpha_{i_*}$ in Theorem~\ref{thm:suff1} can be interpreted in terms of
the single parameter $C$.}
\end{remark}

\begin{proof}
  We use Theorem~\ref{thm:reduction}. Accordingly, it is enough to
  prove that $M(\mu_i)$ has a positive eigenvalue $\lambda$ for some
  $\mu_i$. We will now show that for any {$\mu<0$}, the matrix $M(\mu)$
  has a positive eigenvalue under the given assumptions.
  
  To this end, fix some $\mu<0$, and set $K = -\mu \us \chi>0$. Let  
  \[
  B = M(\mu) + r I,
  \]
  where $r>0$ is chosen large enough so that~$B$ has positive
  diagonal entries.  \com{Note that while both $r$ and $K$ are dependent on $\mu$, $r$ does not depend directly on the value of $K$ (which may be changed by varying parameters other than $\mu$). Note also that $r$ is independent of $\valpha$.} Then since $J$ is Metzler, $B$ is a nonnegative
  matrix.

  In fact, $B$ is an irreducible nonnegative matrix. To see this,
  consider the directed graph of $J$ made of vertices $V_1,\ldots,
  V_N$. The graph $G(B)$ is obtained after augmenting $G(J)$ by a
  vertex corresponding to the first row and column, say $V_0$, and
  adding vertex loops as needed to account for the positive diagonal
  of $B$.  Since $J$ is irreducible, $G(J)$ is strongly connected.
  Moreover, since $\valpha \ne 0$, we conclude that 
  there is a path
  from $V_j$ to $V_0$ for any $j$. Similarly, since $K>0$, there is a
  path from $V_0$ to $V_j$ for any $j$. Hence, $G(B)$ is strongly
  connected and $B$ is irreducible.

  We now claim that 
  \begin{equation}
    \label{eq:claim}
    \lim_{K\to \infty} \rho(B) = \infty
    \quad\text{and}\quad
    \lim_{\alpha_{i_*} \to \infty} \rho(B) = \infty.
  \end{equation}
  The theorem follows from the claim.  Indeed, since $B$ is a
  nonnegative irreducible matrix, $\rho(B)$ is an eigenvalue of $B$ by
  Theorem~\ref{thm:perronfrob}. Hence, by~\eqref{eq:claim}, for large
  enough $K$, we can ensure that $\rho(B)>r$. But then, $\rho(B)-r$ is
  a positive eigenvalue of $M(\mu) = B - r I$. By the same reasoning,
  for large enough $\alpha_{i_*}$, the matrix $M(\mu)$ has a positive
  eigenvalue. Thus, in either case, the stationary state $(\us,\vs)$
  is unstable.

  In the remainder of this proof, we establish the
  claim~\eqref{eq:claim}. To this end, let us \com{denote by $B^t$ the transpose of the matrix $B$} and note that for any 
  integer~$m\ge 1$,
  \[
  \rho(B^t)^m = \rho( (B^t)^m ).
  \]
  Therefore, to prove that the first limit in~\eqref{eq:claim} is
  infinite, it suffices to prove that
  \begin{equation}
    \label{eq:5}
      \lim_{K\to \infty } \rho ((B^t)^m) = \infty 
  \end{equation}
  for some integer~$m\ge 1$, take the $m$th root, and use
  $\rho(B)=\rho(B^t)$.

  To prove~\eqref{eq:5}, we start by observing that the $(i,j)$th
  entry of $(B^t)^m$ is
  \begin{equation}\label{matmult}
    [(B^t)^m]_{ij} 
    = \sum_{i_1}\sum_{i_2}\cdots \sum_{i_{m-1}} [B^t]_{i, i_1} [B^t]_{i_1,i_2} \cdots [B^t]_{i_{m-1},j}
  \end{equation}
  where the sums run over the ranges of the matrix indices $i_1,\dots
  i_{m-1}$. We choose these ranges to be $0,1,\ldots N$ (instead of
  the customary $1,2,\ldots, N+1$) so as to match the indices of the
  vertices $V_0,\ldots, V_N$ of $G(B)$. Since $G(B^t)$ is strongly
  connected, there is a path from $V_i$ to $V_j$ for any $i$ and $j$
  of some length (number of connecting edges) $l$.  Since $B^t$ has
  positive diagonal entries, each vertex of $G(B^t)$ has a loop, and
  consequently, if there is a path of length $l$, then there is a path
  of any length longer than $l$ as well. Hence, there is a number $m$
  such that every two vertices in $G(B^t)$ are connected by a path of
  length~$m$. With this $m$, consider the $(i,0)$th entry of
  $(B^t)^m$, as given by~\eqref{matmult}. Since there is a path
  connecting $V_i$ to $V_0$, say $V_i\to V_{i_1} \to V_{i_2} \cdots
  V_{i_{m-1}} \to V_0$ with $i_{m-1} = N$, we have
  from~\eqref{matmult} that
  \[
    [(B^t)^m]_{i0} 
    \ge
    B_{i_1, i} B_{i_2, i_1} \cdots B_{0, N}
    =
    C_i(K) K,
 \]
 where $C_i(K) \equiv B_{i_1, i} B_{i_2, i_1} \cdots
 B_{i_{m-1},i_{m-2}}$ is a positive and nondecreasing function of~$K$
 (in fact a polynomial in $K$). Let $C(K)$ denote the minimum of such
 $C_i(K)$ for all $0\le i \le N$. Then, the
 minimal row sum $s(B)$ satisfies
 \[
 s( (B^t)^m ) \ge \min_{0\le i\le N}     [(B^t)^m]_{i0} \ge C(K) K,
 \]
 so by Theorem~\ref{thm:sradius}, 
 \[
 \rho( (B^t)^m ) \,\ge \,s( (B^t)^m )\, \ge\, C(K) K.
 \]
 Since $C(K)$ is a nondecreasing positive function of $K$, letting
 $K\to \infty$, we prove~\eqref{eq:5}.

 This proves that the first limit in~\eqref{eq:claim} is infinity. The
 proof that the second limit is also infinity follows in the same
 manner by finding, for each $0 \leq i \leq N$, a path of length $m$
 from $V_i$ to $V_{i^*}$.
\end{proof}

\subsection{Relaxing the irreducibility condition}

As in the previous proof, let $V_1,V_2,\ldots V_N$ be the vertices of
the $G(J)$.  One way to check the \com{second} condition (irreducibility) of
Theorem~\ref{thm:suff1} is to verify that there is a path from any
$V_i$ to $V_j$ for all $i\ne j$. In this subsection, we will give
another condition involving only one path that is often
easier to check.

Let us first recall some further terminology.  
A {\em strongly connected component}
of a graph is a maximal strongly connected subgraph. Note that a
strongly connected graph consists of exactly one strongly connected
component. We consider the following equivalence relation on vertices
of a graph: Two vertices $V_i$ and $V_j$ in a graph are equivalent if
there is a directed path from $V_i$ to $V_j$ \textit{and} a path from
$V_j$ to $V_i$. Then, the equivalence classes created by this
equivalence relation consist exactly of the vertices of the strongly
connected components of the graph. Keeping these in mind, recall the
following definition of~\cite{mmats}.

\begin{definition} 
  The {\em classes} of an $N \times N$ nonnegative matrix $A$ are the
  disjoint subsets $\{i_1,i_2,\ldots,i_k \} \subseteq
  \{1,2,\ldots,N\}$ corresponding to the vertices of the equivalence
  classes of~$G(A)$. We also identify a class with its strongly
  connected component.
\end{definition}

If $A$ is reducible, then  by reordering the vertices, a
permutation matrix $P$ can be found so that
\begin{equation}
  \label{eq:6}
  T = PAP^t
\end{equation}
is a lower block triangular matrix with the blocks consisting of the
classes of $A$ (see~\cite{mmats}).
Since the classes of $A$ are strongly connected, the diagonal blocks
of $T$ are irreducible. Using these facts we can drop the
irreducibility assumption in Theorem~\ref{thm:suff1} and replace it
with a weaker assumption as follows.

\begin{theorem}          \label{thm:suff2}
  Assume that the system~\eqref{eq:GKS} has a positive
  homogeneous steady state solution $(\us,\vs)$ and suppose that
  $J$ and $\valpha \ge 0$
  satisfy the following conditions:
\begin{enumerate}
\item \label{item:i*} There is an \com{$i_*$ with} $1 \leq i_* \leq N$ such that
  \begin{enumerate}
  \item $\alpha_{i_*} > 0$, and 
  \item a directed path from $V_{i_*}$ to $V_N$ in $G(J^t)$ exists.

  \end{enumerate}
\item \label{item:JMetzler} 
   $J$ is Metzler.

\end{enumerate}

Then, if either the product $\chi \us$ or $\alpha_{i_*}$ is
sufficiently large, $(\us,\vs)$ is linearly unstable.
\end{theorem}

\begin{proof}
  As in the proof of Theorem~\ref{thm:suff1}, fix {$\mu<0$}, choose
  $r$ large enough so that $B = M(\mu) + r I$ has positive diagonal
  entries, and consider the graph $G(B^t)$ with vertices
  $V_0,V_1,\ldots, V_N$.

  We proceed by identifying a cycle in the graph.  Since
  $[B^t]_{0,i_*}=\alpha_{i_*}>0$, there is an edge from $V_0$ to
  $V_{i_*}$. Also, by our assumption, there is a path from $V_{i_*}$
  to $V_N$.  Moreover, since $[B^t]_{N0} = K > 0$ there is an edge
  from $V_N$ to $V_0$, thus completing a cycle. In particular, we have
  shown that $V_0, V_{i_*},$ and $V_N$ are part of the same strongly
  connected component in $G(B^t)$.

  Therefore, finding a permutation matrix $P$ as in~\eqref{eq:6}, we
  conclude that there is a triangular matrix $ T^t = P B^t P^t $ with
  a diagonal block $T_j$ corresponding to the strongly connected
  component containing $V_0,V_{i_*},$ and $V_N$. Due to the strong
  connectivity, $T_j$ is irreducible. Now, we can prove that 
  \[
  \lim_{K\to \infty} \rho(T_j) = \infty
  \quad\text{and}\quad
  \lim_{\alpha_{i_*} \to \infty} \rho(T_j) = \infty,
  \]
  by repeating the arguments in the proof of~\eqref{eq:claim}, because
  $T_j$ is a nonnegative irreducible matrix with a structure similar
  to the matrix in~\eqref{eq:claim}.

  Thus, for sufficiently large $K$ or $\alpha_{i_*}$, the triangular
  matrix $T$ has an eigenvalue larger than $r$, and so does $B$.
  Hence $M(\mu) = B - rI $ has a positive eigenvalue.
\end{proof}

\section{Examples}   \label{sec:examples}

In this section, we present several examples of how to use the theory
developed in the previous sections. We consider the
system~\eqref{eq:GKS} with a $\vg(\vv)$ that describes the kinetics of
some CRN linking the chemicals $v_1, v_2,\ldots, v_N$. In all the
examples, $\vg$ is obtained by the laws of mass action kinetics.

\begin{example}[\com{(The linear case)} ]
  Suppose $\vg(\vv)$ is linear, i.e.,
  \begin{equation}
    \label{eq:7}
      \vg(\vv) = A \vv
  \end{equation}
  for some $N\times N$ matrix $A$.
  This limits the network to reactions
  of the form $ v_i \to v_j$ (``$v_i$ produces $v_j$''), or $ v_i \to
  \emptyset,$ (``$v_i$ decays''), but these component reactions can be
  combined arbitrarily.  

  Let us first observe a consequence of mass action kinetics. The
  reaction $v_i \rightarrow v_j$ with a rate constant $k_{ji}$ gives
  rise to an off-diagonal entry $k_{ji}$ in the matrix $A$. Since the
  rate constants are nonnegative, this implies that $A$ must be a Metzler
  matrix.

  If, in addition, we know that $-A$ is a nonsingular $M$-matrix, then
  Proposition~\ref{prop:linear} can be applied to conclude the
  existence of nonnegative steady states.

  To investigate stability, we proceed by studying the paths in the
  CRN and applying our previous results, as exemplified next.

\begin{corollary} \label{cor:example} 
  Assume that~\eqref{eq:GKS} has a positive homogeneous steady state
  solution $(\us,\vs)$.  Assume that the kinetics of a CRN on the
  chemicals $v_1,v_2,...,v_N$ is described by a linear function
  $\vg(\vv)$ which is obtained by the law of mass action kinetics.
  If, for some chemical~$v_{i_*}$, there is a path in the CRN from
  $v_{i_*}$ to $v_N$, and $\alpha_{i_*}>0$, then $(\us,\vs)$ is
  linearly unstable whenever $\chi \us$ or $\alpha_{i_*}$ is
  sufficiently large.
\end{corollary}
\begin{proof}
  We only need to verify the conditions of Theorem~\ref{thm:suff2}.
  View the CRN as a directed graph on vertices $V_1,V_2,\ldots ,V_N$
  with an edge from $V_i$ to $V_j$ if and only if there is a reaction
  producing the chemical $v_j$ from the chemical $v_i$.  This graph,
  except possibly for diagonal loops, is exactly $G(A^t)$, where $A$
  is as in~\eqref{eq:7}.  The linearity of $\vg$ also implies that the
  Jacobian $J$ in~\eqref{eq:J} coincides with $A$.  Hence, the given
  path condition verifies Condition~(\ref{item:i*}) of
  Theorem~\ref{thm:suff2}. We have already seen above that the
  assumptions of mass action kinetics imply that $J=A$ is Metzler,
  thus verifying Condition~(\ref{item:JMetzler}) of
  Theorem~\ref{thm:suff2}.
\end{proof}
\end{example}

\begin{example}[\com{(Dimerization and decay) }]
  One of the simplest chemical reactions is {\em
    dimerization}, or the reaction
  \begin{equation}\label{dime}
    2v_1 \longleftrightarrow v_2.
  \end{equation}
  Additionally, suppose the chemicals $v_1$ and $v_2$ decay at rates
  $\gamma_1 \ge 0$ and $\gamma_2 > 0$, respectively. Let the rate
  constant for the forward reaction $2v_1 \rightarrow v_2$ be $k_1
  > 0$, and of the reverse reaction be $k_2 \geq 0$.  Then,
  according to the law of mass action kinetics, the function $\vec{g}$
  describing the CRN is
  \[
  \vec{g}(\vec{v})
  = \begin{bmatrix} -k_1v_1^2 + k_2 v_2 -\gamma_1 v_1\\  
    k_1v_1^2-k_2v_2 - \gamma_2v_2
  \end{bmatrix}.
  \]
  Finally, assume that a species of density~$u$ produces the
  chemical~$v_1$ at a rate of $\alpha_1>0$, but does not produce $v_2$ ($\alpha_2=0$)
  and consider~\eqref{eq:GKS} in this setting.
  
  Recall from Section~\ref{sec:homog} that
  homogeneous steady states are solutions 
  of~\eqref{eq:ss_condition}, namely,
  \begin{subequations}\label{eq:ss_ex}
  \begin{align} \label{ss_exa}
     \alpha_1 u -k_1v_1^2 + k_2 v_2 -\gamma_1 v_1 
  & = 0, \\
  \label{ss_exb}
   k_1v_1^2-(k_2+\gamma_2)v_2& = 0.
  \end{align}
  \end{subequations}
  Let us verify that this system admits positive steady states. 
  Note that upon adding them, the equations of~\eqref{eq:ss_ex} are equivalent to:
  \begin{subequations}\label{eq:ss_ex2}
  \begin{align} \label{ss_ex2a}
   \alpha_1 u&= \gamma_1 v_1 + \gamma_2 v_2, \\
  \label{ss_ex2b}
   v_2& = \frac{k_1}{k_2+\gamma_2}v_1^2.
  \end{align}
  \end{subequations}
  
  Eliminating $v_2$ from~\eqref{ss_ex2a}
  using~\eqref{ss_ex2b}, we find that $v_1$ satisfies
  \begin{equation}
    \label{eq:9}
    \frac{k_1\gamma_2}{k_2+\gamma_2}v_1^2+\gamma_1v_1-\alpha_1u=0,
  \end{equation}
  a quadratic equation which has a unique positive root whenever
  $u>0$. Letting $r(u)$ denote this positive root, we find the family of
  homogeneous stationary states for this example:
  \[
  u>0 \text{ (arbitrary)}, \quad
  v_1 = r(u)>0, \quad
  v_2 = \frac{k_1}{k_2+\gamma_2}r^2(u)>0.
  \]

  We next turn to determining stability for this example. Clearly
  $\vec{g}$ is not linear, so we cannot apply
  Corollary~\ref{cor:example}. However, we can apply
  Theorem~\ref{thm:suff2}.  The Jacobian of $\vec{g}$ evaluated at a
  positive homogeneous steady state $(\us,\vs)$ is 
  \[ J = \begin{bmatrix}
    -2k_1 \vsii  1 - \gamma_1 & k_2 \\
    2k_1  \vsii 1 & -k_2 - \gamma_2
  \end{bmatrix}
  \] 
  where $\vsii 1$ is the first component of $\vs$.
  Then $J$ is a Metzler matrix with a path in $G(J^t)$ from $v_1$ to
  $v_2$ since $2k_1 \vsii 1 > 0$.  Hence Theorem~\ref{thm:suff2}
  implies that any positive homogeneous steady state $(\us,\vs)$
  of~\eqref{eq:GKS}, with the above $\vg$, can be destabilized for
  large enough values of $\chi \us$ or $\alpha_1$.

  Notice that the Metzler matrix $J$ above has two negative
  eigenvalues (since the trace of $J$ is negative, and the determinant is positive).
  Therefore, the CRN by itself (without the introduction of
  chemotaxis) is stable.  However, once we introduce chemotaxis, the
  arguments above show that with a high enough chemotactic sensitivity
  ($\chi \us$), or a high enough chemical production by the species
  ($\alpha_1$), the stable system can be destabilized.
\end{example}

\begin{example}
  Now we give an example where a hypothesis of Theorems~\ref{thm:suff1} and~\ref{thm:suff2} fails.
  Nonetheless, we can analyze the system by applying 
  Theorem~\ref{thm:reduction}.

  Consider a reaction of the form
  \[
  v_1 + v_2 \,\longleftrightarrow\; v_3
  \]
  where the rate constants of the foward and backward reactions are 
  given by $k_1>0$ and $k_2>0$ respectively. Again,
  assume that a species of concentration $u$ produces only the
  chemical~$v_1$ at a rate $\alpha_1 > 0$ and assume decay rates 
  $\gamma_i$ and production rates $\alpha_i$ of chemical $v_i$ satisfy
$$
\gamma_1>0,\;\; \gamma_2=\gamma_3=0\textrm{ and } \alpha_1=\alpha>0,\;\; \alpha_2=\alpha_3=0.
$$  
  Then the function $\vec{g}$ describing the
  reaction network is
\[
\vec{g}(\vec{v}) = \begin{bmatrix}
-k_1v_1 v_2 +k_2 v_3 -\gamma_1 v_1\\
-k_1v_1v_2 + k_2 v_3 \\
k_1 v_1  v_2 - k_2 v_3
\end{bmatrix}.
\]

To verify that positive steady states exist, we 
look for $u>0$ and $\vv>0$ satisfying
\begin{subequations}\label{eq:ex2ss}
\begin{align} \label{ex2ss1}
-k_1v_1 v_2 +k_2 v_3 -\gamma_1 v_1 + \alpha  u &= 0, \\
\label{ex2ss2}
-k_1v_1v_2 + k_2 v_3 &= 0.
\end{align}
\end{subequations}
Notice that the third equation of $g(v) + \vec{\alpha} u=\vec 0$ 
is omitted because it is the same as the second up to the sign. 
By subtracting~\eqref{ex2ss1} 
from~\eqref{ex2ss2} we obtain the equivalent system
\begin{subequations}\label{eq:ex2simpss}
\begin{align}
\gamma_1 v_1 &= \alpha u,\\
k_1v_1v_2 &= k_2v_3.
\end{align}
\end{subequations}
From these equations it is clear that
\[
u > 0 \text{ (arbitrary)}, \quad
v_1 = \frac{\alpha}{\gamma_1} u, \quad
v_2 >0 \text{ (arbitrary)}, \quad
v_3 = \frac{\alpha k_1}{\gamma_1 k_2} u v_2.
\]
is a positive solution of~\eqref{eq:ex2ss}.

Having verified that positive steady states exist,
 we turn to determine their stability. 
The Jacobian of the function $\vec{g}$ evaluated at a positive homogeneous 
steady state $(\us,\vs)$ is
\[ 
J = \begin{bmatrix}
-k_1v_{*,2} - \gamma_1 & -k_{1} v_{*,1} & k_{2} \\
-k_1v_{*,2} & -k_1 v_{*,1}& k_2 \\
k_1v_{*,2} & k_1 v_{*,1} & -k_2
\end{bmatrix}.
\]
Since $J$ is not Metzler ($J_{12},J_{21}<0$), 
a hypothesis of Theorems~\ref{thm:suff1} and~\ref{thm:suff2} fails.

However, by virtue of Theorem~\ref{thm:reduction}, the linearized
stability of steady states is determined by the eigenvalues of the
matrix
\[ 
M(\mu) = 
\begin{bmatrix}
\mu D& 0 & 0 & -\chi \us \mu \\
\alpha_1 & -k_1 v_{*,2} + \mu \tilde{D}_1 - \gamma_1 & -k_1 v_{*,1} & k_2 \\
0 & -k_1 v_{*,2}  & -k_1v_{*,1} + \mu \tilde{D}_2  & k_2 \\
0 & k_1 v_{*,2} & k_1 v_{*,1} & -k_2 + \mu \tilde{D}_3
\end{bmatrix}
\]
as $\mu$ varies over the spectrum of the Laplace operator.
Fix any $\mu < 0$ and let $K = -\chi \us  \mu.$

First consider the case $K=0$, when there is no chemotactic term.
Then, $M(\mu)$ is block triangular and its eigenvalues are $\mu D<0$
and those of its next diagonal block, which can be expressed as
$$
 \begin{bmatrix}
 - a-d_1 & -b & c \\
 -a& -b-d_2  & c \\
 a & b& -c-d_3 
\end{bmatrix}
$$
using the positive numbers $a=k_1 v_{*,2}, b=k_1 v_{*,1}, c=k_2,
d_1=\gamma_1 - \mu \tilde D_1, d_2=-\mu\tilde D_2,$ and $d_3=-\mu
\tilde D_3.$ Its characteristic equation is
\begin{eqnarray*}
\lambda^3+b_2\lambda^2+b_1\lambda+b_0:=
&&\lambda^3+(a+b+c+d_1+d_2+d_3)\lambda^2\\
&&+\left( a(d_2+d_3)+b(d_1+d_3)+c(d_1+d_2)+d_1d_2+d_1d_3+d_2d_3\right)\lambda \\
&& + (ad_2d_3+bd_1d_3+cd_1d_2+d_1d_2d_3)=0
\end{eqnarray*}
We apply the Routh-Hurwitz criterion~\cite{Gantm59}, which says that
all solutions of the characteristic equation have negative real part
if and only if the following inequalities hold:
$$
b_2>0,\;\; b_0>0\textrm{ and }b_1b_2-b_0>0.
$$
The first two are immediate. The third can be verified after some
calculations: every term in the product $b_1b_2$ is positive, and
every term in $b_0$ cancels some term (but not all terms) in
$b_1b_2$. Hence for any $\mu < 0$ all eigenvalues of $M(\mu)$ have
negative real part.

We next turn to the case where $K >0$. Using cofactor expansion along
the first row, we can compute $\det M(\mu)$ to be
\[
C - K 
  \det 
  \begin{bmatrix}
    \alpha_1 & -k_1 \vsii{2} + \mu \tilde{D}_1 - \gamma_1 & -k_1\vsii{1}\\
    0 & -k_1 \vsii{2}  & -k_1\vsii{1} + \mu \tilde{D}_2\\
    0 & k_1 \vsii{2} & k_1 \vsii{1}
  \end{bmatrix}
\]
where $ C= -\mu D b_0$ is the determinant of the matrix $M(\mu)$ in
the case $K=0$. Clearly $C>0$ is independent of $K$.
Simplifying and identifying the $K$ dependence, 
\[
\det M(\mu)  = C + K \alpha_1 k_1\vsii{2} \mu \tilde{D}_2
\]
where $ K \alpha_1 k_1\vsii{2} \mu \tilde{D}_2 < 0$.  Hence there is a
$K_0>0$ such that $\det M(\mu) <0$ for all $K \ge K_0$.

To conclude, we claim that $M(\mu)$ must have a negative eigenvalue
when $K\ge K_0$. If all eigenvalues of $M(\mu)$ are real, the claim is
obvious.  If not, complex eigenvalues occur in conjugate pairs, so in
order that their product is negative, at least two of the four
eigenvalues must be real.  Furthermore, these two real eigenvalues
must have opposite signs to satisfy $\det M(\mu)<0$.  

Therefore, for all sufficiently large values of~$K$, the steady state
is linearly unstable.
\end{example}

\end{document}